\documentclass[10pt]{amsart}

\usepackage[margin=1.5in]{geometry}                % See geometry.pdf to learn the layout options. There are lots.
\usepackage{graphicx}
\usepackage{amssymb}
\usepackage{epstopdf}
\usepackage{color}
\usepackage{enumerate}
\usepackage{mathtools}
\usepackage{enumitem}
\usepackage{amsmath}
\usepackage{braket}
\usepackage{mathrsfs}  
\usepackage{hyperref}

\usepackage [english]{babel}
\usepackage [autostyle, english = american]{csquotes}
\MakeOuterQuote{"}

\newcommand{\R}{\mathbb{R}}
\newcommand{\N}{\mathbb{N}}
\newcommand{\Z}{\mathbb{Z}}

\newcommand{\Q}{\mathbb{Q}}

\newtheorem{theorem}{Theorem}
\newtheorem{corollary}[theorem]{Corollary}

\newtheorem{fact}{Theorem}[section]

\newtheorem{cor}[fact]{Corollary}
\newtheorem{lemma}[fact]{Lemma}

\newtheorem{obs}[fact]{Observation}
\newtheorem{ex}[fact]{Example}

\newtheorem{defn}[fact]{Definition}

\title{Bounds for Equilibrium States on Amenable Group Subshifts}

\date{26 Sep 2024}
\begin{document}

\maketitle
\begin{center}
\author{C. Evans Hedges}

\end{center}

% Keywords: thermodynamic formalism; equilibrium state; Gibbs measure; amenable group; pressure; subshift; entropy; conformal Gibbs. 

\begin{abstract} We prove a result on equilibrium measures for potentials with summable variation on arbitrary subshifts over a countable amenable group. For finite configurations $v$ and $w$, if $v$ is always replaceable by $w$, we obtain a bound on the measure of $v$ depending on the measure of $w$ and a cocycle induced by the potential.

We then use this result to show that under this replaceability condition, we can obtain bounds on the Lebesgue-Radon-Nikodym derivative $d (\mu_\phi \circ \xi ) / d\mu_\phi $ for certain holonomies $\xi$ that generate the homoclinic (Gibbs) relation. 

As corollaries, we obtain extensions of results by Meyerovitch and Garcia-Ramos and Pavlov to the countable amenable group subshift setting. Our methods rely on the exact tiling result for countable amenable groups by Downarowicz, Huczek, and Zhang and an adapted proof technique from Garcia-Ramos and Pavlov.

\end{abstract}

\section{Introduction}

This paper is concerned with equilibrium states on subshifts over a countable amenable group. In particular, for an arbitrary subshift, given an equilibrium state $\mu$ for a potential with summable variation, we prove Gibbs-like bounds on the measures of finite configurations under a replicability condition. We use this result to prove a novel conformal Gibbs-like bound on the Lebesgue-Radon-Nikodym (LRN) derivative  $d ( \mu \circ \xi ) / d\mu$ for a certain class of Borel isomorphisms $\xi$.  Our results generalize results of Meyerovitch \cite{meyerovitch} and Garcia-Ramos and Pavlov \cite{GRP} and correct an error in the latter paper.

Let $X \subset \mathcal{A}^G$ be a subshift (i.e. a closed and shift-invariant subset) over a countable amenable group $G$ with a finite alphabet $\mathcal{A}$. The dynamics of the system will be induced by the left-translation map $\{ \sigma_g \}_{g \in G}$ where $X$ is required to be compact and $\sigma$-invariant. A potential is a continuous, real valued function $\phi: X \rightarrow \R$. Equilibrium states are $\sigma$-invariant Borel probability measures maximizing the pressure of $\mu$ with respect to $\phi$: $P_\phi(\mu) = h(\mu) + \int \phi d\mu$, where $h(\mu)$ is the classical Kolmogorov-Sinai entropy of $\mu$. Although we make no further assumptions on our subshift, we will require that the potential under consideration have summable variation, which will be defined in Section 3.2. 

There is an extensive history in statistical physics and dynamical systems relating the global property of being an equilibrium state to local properties depending on the potential. Thermodynamic Formalism, at its core, concerns itself with relating these global and local phenomena. Foundational results in this area were obtained by Dobru\v sin in \cite{D}, and Lanford-Ruelle in \cite{LR}, who considered well behaved $\Z^d$-subshifts coupled with a sufficiently regular potential. In this setting, they were able to show that an invariant measure is an equilibrium state if and only if it can be locally characterized by the Gibbs property. 

We say a measure $\mu$ is Gibbs for $\phi$ if it satisfies a conditional probability condition. For a finite $F \Subset G$, we define the $F$-language of $X$, $L_F(X) = \{ w \in \mathcal{A}^F : \exists x \in X : x_F = w \}$, to be the set of all $F$-shape configurations that are legal in $X$. For $w \in L_F(X)$ we can define the extender set of $w$ as in \cite{kass} and \cite{ormespavlovextenderset}, $E_X(w) = \{ \eta \in \mathcal{A}^{F^c} : w \eta \in X \}$, to be the collection of all background configurations for $w$ such that $w \eta \in X$. We say that $\mu$ is Gibbs for $\phi$ if for any configuration $w \in L_F(X)$, and almost every background configuration $\eta \in E_X(w)$, we have 
$$\mu(w || \eta) = \frac{exp( \phi( w\eta))}{\sum_{v \in L_F(X)}  exp( \phi( v \eta)) \cdot 1_X(v \eta) } .$$
The Dobru\v sin theorem and Lanford-Ruelle theorem relating Gibbs measures and equilibrium states have been extended to the countable amenable group subshift setting in \cite{dlramenable1} and \cite{dlramenable2}, where it was shown that  for sufficiently regular subshifts and potentials with summable variation, an invariant measure is an equilibrium state if and only if it is Gibbs for the potential.

In the case where $\phi = 0$, equilibrium states correspond to measures of maximal entropy (MMEs). Parry showed in \cite{parry} that for $\Z$-subshifts of finite type, the MME is unique (and in fact by application of the Lanford-Ruelle theorem, it is Gibbs for $\phi = 0$). In \cite{bowen}, Bowen showed that for an expansive $\Z$-action on a compact metric space satisfying the specification property and a potential with summable variation, there exists a unique equilibrium state. However in the $\Z^d$ setting with $d \geq 2$, Burton and Steif in \cite{burtonmanymme} and \cite{burtonnewmme} were able to construct strongly irreducible subshifts of finite type with non-unique MMEs. Additionally, in the $\Z$ setting, there are trivial examples of subshifts with positive entropy and a unique MME such that the MME is not Gibbs for $\phi = 0$. Take for example, the product of a full shift with the orbit closure of the point $0^\infty 1 0^\infty$, whose unique MME is the product of the unique MME on the full shift with $\delta_{0^\infty}$. %(see Example  \ref{sunnysideexample})}. 

%And in the $\Z$ setting, Ku{\l}aga-Przymus and Lema{\'n}czyk in \cite{mmenotgibbs} constructed a subshift with a unique MME that is not Gibbs. 

While neither uniqueness nor the Gibbs property may be attainable for MMEs for a general subshift, Garcia-Ramos and Pavlov proved in \cite{GRP} that for arbitrary $\Z^d$-subshifts, and any MME, one can obtain bounds on the measures of finite configurations under a replaceability condition. For a subshift $X \subset \mathcal{A}^{\Z^d}$ and $w, v \in L_F(X)$, we say $v$ is replaceable by $w$ if $E_X(v) \subset E_X(w)$. Garcia-Ramos and Pavlov showed that for any MME $\mu$ on a $\Z^d$-subshift, if $v$ is replaceable by $w$ then  $\mu([v]) \leq \mu([w])$. 

The context considered in this paper will combine that of Meyerovitch in \cite{meyerovitch} and Garcia-Ramos and Pavlov in \cite{GRP}. We make no assumptions on the subshift under consideration, and we require that the potential $\phi \in SV(X)$ has summable variation. Our results make use of a class of Borel isomorphisms: for finite configurations $v, w \in L_F(X)$, define $\xi_{v, w}$ pointwise to swap $v$ and $w$ in the $F$ location whenever legal in $X$. Note here that we do not require $\xi_{v, w}$ to be continuous, and in general it is not (see Section 2.2 for a precise definition and further discussion). Our first result can now be stated:

\begin{theorem}\label{sv theorem} Let $G$ be a countable amenable group and $X$ be a $G$-subshift. Let $\phi \in SV(X)$, $\mu_\phi$ an equilibrium state for $\phi$, $F \Subset G$ and $v, w \in L_F(X)$. If $E_X(v) \subset E_X(w)$ then 
$$\mu_\phi([v]) \leq \mu_\phi([w]) \cdot  \sup_{x \in [v]} exp\left(  \sum_{g \in G} \phi(\sigma_g(x)) - \phi(\sigma_g(\xi_{v, w}(x)))  \right).$$
\end{theorem}

We note here that this inequality holds whenever the conclusion of the Lanford-Ruelle Theorem also holds (Theorem \ref{lr theorem} below), or even more generally when the conclusion of Theorem \ref{homoclinic theorem} holds (as noted in Observation \ref{thm3 obs}). In particular, the equation immediately holds for all subshifts of finite type and potentials with summable variation, the novelty here is that we require no assumptions on the structure of the subshift $X$. It is in this general case where we must discuss extender sets, as in \cite{GRP} and \cite{meyerovitch}.

In general this supremum may be hard to compute. However, an immediate corollary in the locally constant case allows us to easily compute this bound when $v$ and $w$ agree on a sufficient boundary. First, for finite $H \Subset G$, we call $\phi$ an $H$-potential when if   $x_H = y_H$, then $\phi(x) = \phi(y)$. In particular, this means for any $v \in L_F(X)$ with $H \subset F$, we can write $\phi(v)$ unambiguously to mean $\phi(x)$ for any $x \in X$ such that $x_F = v$ since for all $x, y \in X$ with $x_F = y_F = v$, $\phi(x) = \phi(y)$. We also denote $H^\pm = H \cup H^{-1}$. 
 
\begin{corollary}\label{locally constant cor} Let $H, F \Subset G$, $v, w \in L_F(X)$, and $\phi$ be an $H$-potential. Suppose that $E_X(v) \subset E_X(w)$ and for all $g \in F^c H^\pm \cap F$, $v_g = w_g$. Then for any equilibrium state $\mu_\phi$ for $\phi$, 
$$\mu_\phi([v]) \leq \mu_\phi([w])  \cdot exp \left( \sum_{g \in F \backslash F^c H^{-1}} \phi ( \sigma_g (v)) - \phi ( \sigma_g(\xi_{v, w}(w))) \right).$$
\end{corollary}

As another immediate corollary, by letting $\phi = 0$ we extend Theorem 4.4 of \cite{GRP} by Garcia-Ramos and Pavlov to arbitrary countable amenable groups. Although the theorem in \cite{GRP} is stated for countable, finitely generated, torsion free, amenable groups, due to an unfortunate error their proof technique only applies immediately to the case where $G = \Z^d$. The error is the false assertion that for any torsion free, finitely generated, countable amenable group $G = \braket{g_1, \dots, g_k}$, the subgroup generated by $\braket{g_1^n, \dots, g_k^n}$ has finite index in $G$. This is known to be false and can be shown not to hold in a variety of examples, including the Lamplighter group. 

%(although with minimal effort their proof technique can be extended to the case where $G$ is countable, residually finite, and amenable).

$\;$

%In addition to the classical definition of Gibbs in the sense of Dobru\v sin, Lanford, and Ruelle, another fruitful approach has been to consider a measure that is conformal Gibbs for a potential. This definition is stated in terms of the LRN derivative $d (\mu \circ \xi ) / d\mu$ for a certain class of Borel isomorphisms $\xi$ {\it and will be explored in detail in Section 4}. It was shown by Borsato and MacDonald in \cite{conformalisgibbs} that for subshifts over a countable group and a potential $\phi$, a measure is Gibbs for $\phi$ if and only if it is conformal Gibbs for $\phi$. 

In addition to the classical definition of Gibbs in the sense of Dobru\v sin, Lanford, and Ruelle, another fruitful approach has been to consider a measure that is conformal Gibbs for a potential. It was shown by Borsato and MacDonald in \cite{conformalisgibbs} that for subshifts over a countable group and any potential $\phi$, a measure is Gibbs for $\phi$ if and only if it is conformal Gibbs for $\phi$ (see Section 4 for a precise definition). As a consequence, this means that $\mu_\phi$ is Gibbs for $\phi$ if and only if, for every Borel isomorphism of the form $\xi_{v, w}$ and for $\mu_\phi$-a.e. $x \in X$, 
$$\frac{ d (\mu_\phi \circ \xi_{v, w} ) }{d \mu}(x) = exp \left( \sum_{g \in G} \phi( \sigma_g(\xi_{v, w}(x))) - \phi( \sigma_g(x)) \right). $$
Since, in general, equilibrium states are not necessarily Gibbs, this equality cannot always hold. In the general subshift setting of this paper, little can be said of this LRN derivative.  

%In fact, as mentioned above, in \cite{mmenotgibbs} Ku{\l}aga-Przymus and Lema{\'n}czyk constructed a $\Z$-subshift with a unique MME that is not Gibbs (for $\phi = 0$), meaning that there must exist some $\xi_{v, w}$ such that the above equality does not hold. 

In \cite{meyerovitch}, Meyerovitch showed that for a general $\Z^d$ subshift $X$ and potential $\phi$ with $d$-summable variation, if $\mu_\phi$ is an equilibrium state for $\phi$ and if $E_X(v) = E_X(w)$, then $\frac{ d (\mu_\phi \circ \xi_{v, w} ) }{d \mu}$ satisfies the equation above. In the language of Meyerovitch, $\mu_\phi$  must be $(\mathfrak{T}_X^0, \psi_\phi)$-conformal.

%As an extension of the original Lanford-Ruelle theorem in the $\Z^d$ setting, Meyerovitch has shown in \cite{meyerovitch} that for any subshift and any equilibrium state $\mu_\phi$ for a potential $\phi$ with $d$-summable variation, $\mu_\phi$ must satisfy the Gibbs LRN derivative equations for a subclass of the relevant Borel isomorphisms $\xi$. In particular, $\mu_\phi$ must be $(\mathfrak{T}_X^0, \psi_\phi)$-conformal, which will be defined in Section 4. 

%In Section 4, we will use Theorem \ref{sv theorem} to obtain the following bound on the LRN derivative of $\mu_\phi \circ \xi_{v, w}$ with respect to $\mu_\phi$, showing a conformal Gibbs-like result.  

In Section 4, we will use Theorem \ref{sv theorem} to obtain the following bound on this LRN, showing a conformal Gibbs-like result.  
\begin{theorem}\label{homoclinic theorem} Let $F \Subset G$, $v, w \in L_F(X)$, $\phi \in SV(X)$, and $\mu_\phi$ an equilibrium state for $\phi$. If $E_X(v) \subset E_X(w)$, then $\mu_\phi \circ \xi_{v, w}$ is absolutely continuous with respect to $\mu_\phi$ when restricted to $[w]$ and for $\mu_\phi$-almost every $x \in [w]$, 
$$\frac{ d (\mu_\phi \circ \xi_{v, w} ) }{d \mu_\phi }(x) \leq exp \left( \sum_{g \in G} \phi( \sigma_g(\xi_{v, w}(x))) - \phi( \sigma_g(x)) \right).$$
\end{theorem}

{We have become aware that the Corollary \ref{conformal cor} has been proven in even greater generality in the sofic group setting in \cite{barbieri-meyerovitch}. However we recover the fact in the countable amenable group setting as an easy corollary of Theorem \ref{homoclinic theorem}.  
\begin{corollary}\label{conformal cor} Let $X \subset \mathcal{A}^G$ be a subshift over a countable amenable group $G$, let $\phi \in SV(X)$ be a potential with summable variation, and let $\mu_\phi$ be an equilibrium state for $\phi$. Then $\mu_\phi$ is $(\mathfrak{T}_X^0, \psi_\phi)$-conformal. 
\end{corollary}

The structure of this paper is as follows. We begin with Section 2 on the relevant preliminaries, discussing countable amenable groups and their relevant properties. We then formally introduce subshifts over a countable amenable group, and discuss their thermodynamic formalism. Finally, we discuss equilibrium measures, the Gibbs property, and their relationship.  

In Section 3 we prove Theorem \ref{sv theorem}, beginning with a lemma using Downarowicz, Huczek, and Zhang's exact tiling result from \cite{dhz} to generate a sufficiently sparse almost partition of a given group $G$. After some preliminary lemmas in the subshift setting, we prove Theorem \ref{sv theorem} and conclude Corollary \ref{locally constant cor} in the locally constant case. 

Finally, Section 4 is concerned with the conformal Gibbs perspective where we formally introduce the concept and relevant definitions. We then prove Theorem \ref{homoclinic theorem}, relate it to the results of Meyerovitch, and conclude by extending Theorem 3.1 and Corollary 3.2 of \cite{meyerovitch} to the countable amenable subshift setting.

%
%
%$\;$
%
%$\;$
%
%$\;$
%
%
%
%
%In this paper, under no assumption on the structure of the subshift, we show numerous Gibbs-like inequalities under a replaceability condition. Our results deepen the extensive literature relating equilibrium states with Gibbs-like properties. 
%
%In particular, our novel results include, when a word $v$ is replaceable by $w$, bounds on the measure of words $v$ and $w$, as well as a bound on the LRN derivative of $\mu_\phi \circ \xi_{v, w}$ restricted to $[w]$. We highlight here that Corollary \ref{locally constant cor} provides an easily computable, explicit bound on the measure of cylinder sets for locally constant potentials. 
%
%As corollaries to these results, we extend the theorem of Garcia-Ramos and Pavlov regarding MMEs as well as the $(\mathfrak{T}_X^0, \psi_\phi)$-conformal result of Meyerovitch to the countable amenable group setting. 
%
%
%
%

\subsection*{Acknowledgements} This work was conducted during the author's PhD studies under the supervision of Dr. Ronnie Pavlov. The author would like to thank Dr. Pavlov for all of his assistance throughout the research and writing process. We would also like to thank the referee for their helpful report.

\section{Preliminaries}

\subsection{Countable amenable groups.} Let $G$ be a countable group and denote the identity of $G$ by $e$. We use the notation $K \Subset G$ to indicate that $K$ is a finite subset of $G$. A {\bf F\o lner sequence} for $G$ is a collection of finite subsets $\{ F_n \}$ of $G$ such that $G = \bigcup_{n \in \N} F_n$ and for all $K \Subset G$,  $\lim_{n \rightarrow \infty} |K F_n \triangle F_n| / |F_n| = 0$. A countable group $G$ is called {\bf amenable} if there exists a F\o lner sequence in $G$. 

For a given F\o lner sequence $\{ F_n \}$ we say the sequence is {\bf tempered} if there exists some $C > 0$ such that for all $n > 0$, $\left| \bigcup_{k < n} F_k^{-1} F_n \right| \leq C |F_n |.$ For any F\o lner sequence, there exists a subsequence that is tempered, see Proposition 1.5 of \cite{LINDENSTRAUSS} for a proof of this fact. A deeper discussion of F\o lner sequences and their relevant properties can be found in Chapter 4 of \cite{amenablechapter}, but for our purposes we note that for any amenable group $G$, there exists a tempered F\o lner sequence that can be taken such that $n!$ divides $|F_n|$ for every $n \in \N$. 

Similar in spirit we can define a relative almost-invariance: for any finite $F, T \Subset G$ and $\epsilon > 0$, we say $T$ is {\bf right $(K, \epsilon)$-invariant} if $|TK \triangle T| / |T| < \epsilon$. We can equivalently say $G$ is amenable if for every finite $K \Subset G$ and every $\epsilon > 0$, there exists some finite $T \Subset G$ that is right $(K, \epsilon)$-invariant.

In our setting we will also be interested in a sense of sparseness of sets, for this we define: 
\begin{defn} For any $S \subset G$, $F \Subset G$, we say $S$ is {\bf left $F$-sparse} if for all distinct $s, s' \in S$, 
$$ sF  \cap s' F  = \emptyset.$$
\end{defn}

Note here by a trivial application of the definition we know for any $F \subset H \Subset G$ and $S \subset G$. If $S$ is left $H$ sparse, then $S$ is left $F$ sparse. We will now define the right (left resp.) $H$-interior and $H$-boundary of $F$ for $F, H \Subset G$.

\begin{defn}  The {\bf right $H$-interior} of $F$, denoted by $Int_H(F)$, is defined by: 
$$Int_H(F) = \bigcap_{h \in H} Fh = \{ f \in F : \forall h \in H,  fh \in F \}.$$

The {\bf right $H$-boundary} of $F$, denoted $\partial_H(F)$ is all elements of $F$ not in the $H$-interior. Precisely: $$\partial_H(F) = F \backslash Int_H(F) = \{ f \in F : \exists h \in H \text{ s.t. } fh \notin F \}.$$

The left $H$-interior/boundary of $F$ are defined similarly. 
\end{defn}

The proof of our main theorem will also require taking advantage of a result by Downarowicz, Huczek, and Zhang regarding exact tilings of $G$ from \cite{dhz}. For a countable amenable group $G$, a {\bf finite tiling} $\mathcal{T} = \{ T_i : 1 \leq i \leq k \}$ is a collection of tiles such that $\bigcup_{i \leq k} T_i = G$ and for each $T_i \in \mathcal{T}$, there exists a finite shape $S_i \Subset G$, and a collection of centers $C(S_i) \subset G$ such that $T_i =  S_i C(S_i)$. In particular, a tiling $\mathcal{T}$ is called {\bf exact} if for distinct $c_1 , c_2 \in C(S_i)$, $S_i c_1 \cap S_i c_2 = \emptyset$. We also note it can be assumed that $e \in S$ for every shape. We can now state their result: 

\begin{fact}[Theorem 4.3 \cite{dhz}] Given any infinite, countable amenable group $G$, any $\epsilon > 0$, and any finite $K \Subset G$, there exists an exact tiling $\mathcal{T}$ where each shape is right $(K, \epsilon)$-invariant. 
\end{fact}

Their result in fact included statements regarding the entropy of the tiling space which we have omitted since they are not necessary for the purposes of this paper.

\subsection{$G$-subshifts} Let $G$ be a countable amenable group. Let $\mathcal{A}$ be a finite set, called the alphabet, endowed with the discrete topology. Our configuration space is the set of functions $x: G \rightarrow \mathcal{A}$, which we denote $\mathcal{A}^G$, endowed with the product topology. For any $H \subset G$ and $x \in \mathcal{A}^G$, we denote the restriction of $x$ to $H$ by $x_H$. 

We define a $G$-action of homeomorphisms on $\mathcal{A}^G$ by the {\bf left-translation map}: for all $g \in G$, and $x \in \mathcal{A}^G$, we define pointwise $\sigma_g(x)_h = x_{gh}$. The set $\mathcal{A}^G$ and the collection of left-translations $(\sigma_g)_{g \in G}$ together form a topological dynamical system $(\mathcal{A}^G, \sigma)$ which we call the full $G$-shift on $\mathcal{A}$. A {\bf $G$-subshift} is a subset $X \subset \mathcal{A}^G$ that is closed in the product topology and is $\sigma$-invariant (i.e. for all $g \in G$, $\sigma_g(X) \subset X$). When $G$ is clear, we will refer to a $G$-subshift as just a subshift.

For any finite $F \Subset G$, and any $w \in \mathcal{A}^F$, we call $w$ a {\bf configuration} of shape $F$. For a subshift $X$ and a finite configuration $w$ of shape $F \Subset G$, we say $w$ is in the {\bf language of $X$} (or $w$ is legal in $X$) if there exists some $x \in X$ such that $x_F = w$. We call $L_F(X) = \{ x_F : x \in X \}$ the {\bf$F$-language} of $X$, and $L(X) = \bigcup_{F \Subset G} L_F(X)$ the language of $X$. 

For any $F \Subset G$, and $w \in \mathcal{A}^F$, we define the {\bf cylinder set} of $w$ as follows: 
$$[w] = \{ x \in \mathcal{A}^G : x_F = w \}.$$
In particular, the cylinder sets form a basis for the product topology on $\mathcal{A}^G$. Cylinder sets are frequently taken intersected with a subshift, which will be clear by context. 

Finally we note here that the product topology on $\mathcal{A}^G$ is metrizable, and for any F\o lner sequence $\{ F_n \}$, the metric: 
$$d(x, y) = 2^{ - \min \{ n \in \N : x_{F_n} = y_{F_n} \} }$$
induces the product topology. This metric serves primarily to establish that $G$-subshifts are expansive. 

For any disjoint $H, K \subset G$ and any $x \in \mathcal{A}^H$, $y \in \mathcal{A}^K$, we define the {\bf concatenation} of $x$ and $y$ by $xy \in \mathcal{A}^{H \cup K}$ such that $(xy)_H = x$ and $(xy)_K = y$. We also define for any $v \in L_F(X)$,  the {\bf extender set of $v$}, $E_X(v)$, to be the collection of all background configurations for which $v$ is legal. Specifically, $E_X(v) = \{ \eta \in \mathcal{A}^{G \backslash F} : v\eta \in X \}$. For any $v, w \in L_F(X)$ we say $v$ is {\bf replaceable} by $w$ if $E_X(v) \subset E_X(w)$.

We will now define the following functions that will be useful in the proof of our main results and essential to our discussion of the conformal Gibbs corollaries. 
\begin{defn} For any $F \Subset G$, $v, w \in L_F(X)$, we define $\xi_{v, w} : X \rightarrow X$ through the following cases. 
\begin{itemize}
\item For $x \in [v]$, if the concatenation $w x_{F^c} \in X$ then $\xi_{v, w}(x) = w x_{F^c}$; otherwise, $\xi_{v, w}(x) = x$.
\item For $x \in [w]$, similarly define $\xi_{v, w}(x)$. 
\item Otherwise, let $\xi_{v, w}(x) = x$.
\end{itemize}
\end{defn}

Note this is exactly switching $v$ and $w$ in the $F$-location whenever the resulting point is still in $X$. It will be noted in Section 4 that these functions are in fact Borel isomorphisms that generate the homoclinic (Gibbs) relation.

$\;$

For a given subshift $X \subset \mathcal{A}^G$, will denote the collection of $\sigma$-invariant probability measures on $X$ by $\mathcal{M}_\sigma(X)$. The existence of such measures is guaranteed by the fact that $G$ is amenable. For a given $\mu \in \mathcal{M}_\sigma(X)$, for any collection of legal finite configurations $W \subset L(X)$, we will use $\mu(W) $ to mean $\mu( \bigcup_{w \in W} [w] )$.  

We say $\mu \in \mathcal{M}_\sigma(X)$ is {\bf ergodic} if for all measurable $A \subset X$, for all $g \in G$, if $\mu( A \triangle \sigma_g^{-1} A) = 0$, then $\mu(A) \in \{ 0, 1\}$. $\mathcal{M}_\sigma(X)$ is convex and compact under the weak-$*$ topology, in fact it forms a Choquet simplex whose extreme points are exactly the ergodic measures.

\subsection{Thermodynamic formalism}

The theory of thermodynamic formalism bridges the gap between microscopic and macroscopic descriptions of systems with many interacting particles, extending concepts of statistical mechanics to symbolic dynamics. Gibbs measures are central to this framework, enabling the analysis of global statistical properties derived from local interactions in a broad range of dynamical systems. This section delves into the thermodynamic formalism for countable amenable group actions on finite alphabet subshifts, examining topological pressure and its connection to statistical physics, the construction of partition functions, and the characterization of equilibrium states.

To begin, we must define topological pressure of a given potential over our subshift $X$. We first let $\phi : X \rightarrow \R$ be a real valued, continuous function which we will call a {\bf potential}. We will denote the set of all potentials on a subshift $X$ by $C(X)$. For any $F \Subset G$ we define the {\bf$F$-Birkhoff sum} of $\phi$: $\phi_F = \sum_{g \in F} \phi \circ \sigma_g$. Finally, for any open cover $\mathcal{U}$ of $X$, any $F \Subset G$, we define the open cover $\mathcal{U}^F = \bigvee_{f \in F} \sigma_f^{-1} \mathcal{U} =  \{ \bigcap_{f \in F} \sigma_f^{-1}(U_f) : U_f \in \mathcal{U} \}$. We use this notation to define the following partition function:

\begin{defn} We define a {\bf partition function} for any $F \Subset G$ and any open cover $\mathcal{U}$ of $X$: 
$$Z_F(\phi, \mathcal{U}) = \inf \left\{ \sum_{u \in \mathcal{U}'} exp \left( \sup_{x \in u} \phi_F(x) \right) : \mathcal{U}' \text{ is a subcover of } \mathcal{U}^F \right\} .$$
\end{defn}

We can define the topological pressure of $\phi$ with respect to a given open cover $\mathcal{U}$ to be 
$$P(\phi, \mathcal{U}) = \lim_{n \rightarrow \infty} |F_n|^{-1}  \log Z_{F_n}(\phi, \mathcal{U})$$
for any F\o lner sequence $\{ F_n \}$. This limit is guaranteed to exist and does not depend on the choice of F\o lner sequence, further discussion of which can be found \cite{bufetov} and \cite{Gurevich}. We can now define the following: 
\begin{defn} The {\bf topological pressure of $\phi$} is then 
$$P_{top}(\phi) = \sup_{\mathcal{U}} P(\phi, \mathcal{U}).$$
\end{defn}

Since in the subshift setting, the dynamical system $(X, \sigma)$ is expansive, the above supremum is attained at an open cover $\mathcal{U}$ of diameter less than or equal to the expansiveness constant. In particular, we can compute topological pressure by: 
$$P_{top}(\phi) = \lim_{n \rightarrow \infty} |F_n|^{-1} \log \sum_{w \in L_{F_n}(X)} exp \left( \sup_{x \in [w]} \phi_{F_n}(x) \right) . $$

In addition to the topological pressure, for a given invariant measure $\mu \in \mathcal{M}_\sigma(X)$, we can define the pressure of $\mu$ with respect to $\phi \in C(X)$ as follows: 
$$P_\phi(\mu) = h(\mu) +\int \phi d\mu$$
where $h(\mu)$ is the Kolmogorov-Sinai entropy of the invariant  probability measure $\mu$.  As shown by Ollagnier and Pinchon \cite{amenablevariational}, in the countable amenable subshift setting, the variational principle holds. In particular, for all $\phi \in C(X)$, 
$$P_{top}(\phi) = \sup_{\mu \in \mathcal{M}_\sigma(X)} P_\phi(\mu) .$$

In statistical physics, equilibrium states correspond with probability measures on the state space that minimize the Gibbs free energy of the system. Up to a multiplicative constant, the free energy of an invariant measure corresponds with the negative pressure, and so we similarly define an equilibrium state as follows: 

\begin{defn} We say $\mu \in M_\sigma$ is an {\bf equilibrium state} for $\phi$ if it attains the variational principle supremum, i.e.,  
$$P_{top}(\phi) = h(\mu) + \int \phi d\mu .$$
\end{defn}

When $h$ is upper semicontinuous (as in the case for expansive dynamical systems), we know the collection of equilibrium states for a given $\phi$ is non-empty, compact under the weak-$*$ topology, and convex, where the extreme points are exactly the ergodic equilibrium states.

\subsection{Equilibrium measures and the Gibbs property}

A rich theory has developed around relating equilibrium states and measures with the Gibbs property (for some definition of Gibbs property). In the classical results of Dobru\v sin, Lanford, and Ruelle, the Gibbs property can be defined in terms of conditional probabilities as follows: 

\begin{defn}\label{dlr gibbs} For a subshift $X \subset \mathcal{A}^{\Z^d}$ and a potential $\phi \in C(X)$, we say a measure $\mu$ is {\bf Gibbs} for $\phi$ if for all $F \Subset \Z^d$, for all $w \in L_F(X)$, and for $\mu$-almost every $\eta \in \mathcal{A}^{\Z^d \backslash F}$ such that $w \eta \in X$, 
$$\mu( w || \eta ) = \frac{ exp ( \phi( w \eta )) }{\sum_{v \in L_F(X)} exp( \phi( v \eta )) \cdot 1_X(v\eta)}$$
\end{defn}

In other words, we say a measure is Gibbs for $\phi$ when the probability of a local configuration, $w$, given a background configuration, $\eta$, can be computed in the typical way one computes Gibbs measures in the finite case. The results of Dobru\v sin combined with those of Lanford-Ruelle show equivalence to the Gibbs property and being an equilibrium state under certain assumptions on subshift. 

%\begin{defn} A subshift $X \subset \mathcal{A}^{\Z^d}$ satisfies {\bf property (D)} if there exist increasing sequences of subsets of $\Z^d$,  $(\Lambda_n), (M_n)$ such that for each $n \in \N$, $\Lambda_n \subset M_n$, $\lim_{n \rightarrow \infty} |M_n| / | \Lambda_n | = 1$, and for any $x, y \in X$, there exists $z_n \in X$ such that 
%$$x_{\Lambda_n} = (z_n)_{\Lambda_n} \text{ and } \; y_{\Z^d \backslash M_n} = (z_n)_{\Z^d \backslash M_n}$$
%\end{defn}

The following theorem involves a technical condition called Property (D), which is described in \cite{D}.

\begin{fact}[Dobru\v sin \cite{D}] Let $X \subset \mathcal{A}^{\Z^d}$ be a subshift satisfying Property (D) and $\phi \in C(X)$ be a potential with $d$-summable variation. If $\mu \in \mathcal{M}_\sigma(X)$ is an invariant probability measure that is Gibbs for $\phi$, then $\mu$ is an equilibrium state for $\phi$. 
\end{fact}

%\textcolor{blue}{Here, Property (D) is a technical condition that is described in \cite{D}.  }

\begin{fact}\label{lr theorem}[Lanford-Ruelle \cite{LR}] Let $X \subset \mathcal{A}^{\Z^d}$ be a subshift of finite type and $\phi \in C(X)$ be a potential with $d$-summable variation. If $\mu$ is an equilibrium state for $\phi$, then $\mu$ is Gibbs for $\phi$. 
\end{fact}

Subshifts of finite type, SFTs, are an important and well studied class of subshifts that are not defined here. Analogous results have since been shown in the countable amenable subshift setting \cite{dlramenable1} \cite{dlramenable2}. Note that these statements rely on relatively strong assumptions on both the subshift and the potential. When no assumptions are made of the potential, little can be said about the equilibrium state. In fact, by upper semicontinuity of the entropy map, it can be shown that for any subshift $X \subset \mathcal{A}^G$, and any ergodic state $\mu \in \mathcal{M}_{\sigma}(X)$, there exists a potential $\phi \in C(X)$ for which $\mu$ is the unique equilibrium state \cite{jenkinson}. It is therefore quite natural to retain some regularity assumptions on the class of potentials under consideration.

\section{Proof of Theorem \ref{sv theorem}}

For our results, we will impose a natural regularity assumption on our potential $\phi$, but impose no restriction on the subshift $X \subset \mathcal{A}^G$. Our proof of Theorem \ref{sv theorem} will be adapted from the proof technique of Garcia-Ramos and Pavlov in \cite{GRP}, which involves replacing $v$ with $w$ along a sufficiently sparse and regular grid in $G$. Lemma \ref{almost partition} allows us to construct a partition of $G$ from which we may choose the appropriate grid. We will then prove a few technical lemmas, and finally we will show our main result and conclude with a few notable corollaries.

\subsection{Sufficiently sparse almost partitions}

The following lemma allows us to generate a finite $\epsilon$-almost partition of $G$ where each part is sufficiently sparse relative to some fixed $F \Subset G$. 

\begin{lemma}\label{almost partition} For any $F \Subset G$, $\epsilon > 0$, and F\o lner sequence $\{ F_n \}$, there exists a finite collection $\mathcal{P} = \{ P_i : 1 \leq i \leq N \}$ of pairwise disjoint, left $F$-sparse subsets of $G$ and a subsequence $\{ F_{n_k} \}$ such that 
$$\liminf_{k \rightarrow \infty} \frac{| \bigcup_{P \in \mathcal{P}}P \cap F_{n_k}|}{|F_{n_k}|} \geq 1-\epsilon .$$
\end{lemma}

In lieu of a complete and technical proof, we provide an outline of how this result can be shown. The result can be viewed as a weakened reformulation of a quasitiling result of Ornstein and Weiss in \cite{ornsteinweiss}, or an application of the exact tiling results in \cite{dhz}. In essence, one can construct a quasitiling (or an exact tiling) of $G$ using a finite collection of shapes that are sufficiently $F$-invariant. The $F$-invariance ensures that we can consider only the $F$-interiors of these shapes and maintain a $(1-\epsilon)$-covering of $G$. We then let $P \in \mathcal{P}$ represent all the shifts of a particular location of an element in a specified shape. 
This allows us to ensure that $\mathcal{P}$ is left $F$-sparse since we have restricted to the $F$-interiors of the relevant shapes.

\subsection{Subshift lemmas}

We begin by describing our regularity constraints imposed on the potential $\phi$. 

\begin{defn} The {\bf $F$-variation} of $\phi$ for any $F \subset G$ is defined as
$$Var_{F}(\phi) = \sup \{ \phi(x) - \phi(y) : x, y \in X \text{ and } x_F = y_F \} . $$
\end{defn}

\begin{defn} We say $\{ E_n \}$ is an {\bf exhaustive sequence} in $G$ if $E_1 \subset E_2 \subset \dots$ and $G = \bigcup_{n \in \N} E_n$. 
\end{defn}

\begin{defn} We say $\phi$ has {\bf summable variation according to the exhaustive sequence} $\{ E_n \}$ if 
$$\sum_{n \in \N } |E_{n+1}^{-1} \backslash E_{n}^{-1} | \cdot Var_{ E_n }(\phi) < \infty . $$
\end{defn}

Additionally, we will say $\phi \in C(X)$ has {\bf summable variation} if $\phi$ has summable variation according to some exhaustive sequence. 

For a fixed exhaustive sequence $\{ E_n \}$, we define $SV_{\{ E_n \}}(X)$ to be the collection of all potentials with summable variation according to $\{ E_n \}$. We also let $SV(X) = \bigcup_{\{ E_n \}} SV_{\{ E_n \}}(X)$ denote the collection of all potentials with summable variation according to some exhaustive sequence. 

It is worth noting here that when $G = \Z^d$, we can take $E_n = \{ k \in \Z^d  : |k_i| \leq n \}$ and summable variation according to this sequence corresponds to $d$-summable variation in the typical sense.

\begin{defn} For an exhaustive sequence $\{ E_n \}$, we define the {\bf summable variation norm} on $SV_{\{ E_n \}}(X)$ by: 
$$|| \phi ||_{SV_{\{ E_n\}}} = 2 |E_1| \cdot || \phi||_\infty + \sum_{n \in \N } |E_{n+1}^{-1} \backslash E_{n}^{-1} | \cdot Var_{ E_n }(\phi) . $$
\end{defn}

\begin{obs} If $\phi$ has summable variation according to some exhaustive sequence $\{ E_n \}$, then for any $F \Subset G$ such that $e \in F$, $\phi$ has summable variation according to exhaustive sequence $\{ E_n F \}$. 
\end{obs}

\begin{proof} First let $H \Subset G$ and note $ HF^{-1}$ is also a finite set. We therefore have some $N \in \N$ such that for all $n \geq N$, $ H F^{-1} \subset E_n$ and so $H \subset E_n F$. It immediately follows that $\{ E_n F \}$ is an exhaustive sequence. 

$\;$

Since $e \in F$ by assumption, we have $E_n \subset E_n F$. We now see that 
$$\sum_{n \in \N} |E_{n+1} F \backslash E_n F| Var_{E_nF}(\phi) \leq |F| \sum_{n \in \N} |E_{n+1}  \backslash E_n | Var_{E_n}(\phi) < \infty . $$
\end{proof}

We now restate relevant definitions and lemmas from Garcia-Ramos and Pavlov in \cite{GRP}. 

\begin{defn} For $v, u \in L(X)$ we define $O_v(u) = \{ g \in G : \sigma_g(u)_F  = v \}$. 
\end{defn}

In particular, $O_v(u)$ represents the indices where $v$ occurs in $u$ and can be read as "occurrences of $v$ in $u$." We will now define a replacement function that for a given finite configuration $u$, replaces occurrences of $v$ with $w$ in specified locations.

\begin{defn} Let $F, H \Subset G$ and fix any $v, w \in L_F(X)$ and $u \in L_H(X)$. Let $S \subset O_v(u)$ be an $F$-sparse set of occurrences of $v$ in $u$. We now define $R_u^{v \rightarrow w}(S) = u'$ as follows: 
\begin{itemize}
\item For $s \in S$, $f \in F$, let $u'_{sf} = w_f$, and 
\item for all other $g \in H \backslash SF$, let $u'_g = u_g$. 
\end{itemize}
\end{defn}

Since $S$ is an $F$-sparse subset of $O_v(u)$, we know $u'$ is well defined and uniquely determined. Note here that $u'$ is exactly the configuration obtained by replacing $v$ with $w$ in the $S$-locations. 

We remind the reader that for any finite configuration $v \in L_F(X)$, the extender set of $v$ is defined as $E_X(v) = \{ \eta \in \mathcal{A}^{F^c} : v \eta \in X \}$. Note here whenever $E_X(v) \subset E_X(w)$, then $R_u^{v \rightarrow w}(S) \in L(X)$ for any left $F$-sparse set $S \subset O_v(u)$.

%\begin{defn} For any $v, u \in L(X)$,  $g \in O_v(u)$, and $w$ with the same shape as $v$, define $R_u^{v \rightarrow w}(g) = u'$ where $u'$ is the word obtained by replacing $v$ with $w$ in $u$ at location $g$. 
%\end{defn}
%
%Note here that when $v, w \in L_F(X)$ and $S$ is left $F$-sparse, we may simultaneously replace every instance of $v$ by $w$ at every location $s \in S$. For a left $F$-sparse set $S \subset O_v(u)$, we define $R_u^{v \rightarrow w}(S)$ to be the word obtained by replacing all occurrences of $v$ in $u$ at locations in $S$ with $w$. 

\begin{lemma}\label{lemma 4.2}[Lemma 4.2 \cite{GRP}] For any $F$, $v, w \in L_F(X)$ with $v \neq w$, and left $F$-sparse set $T \subset O_v(u)$, $R_u^{v \rightarrow w}$ is injective on subsets of $T$. 
\end{lemma}

\begin{lemma}\label{lemma 4.3}[Lemma 4.3 \cite{GRP}] For any $F$ and $v, w \in L_F(X)$, any left $F$-sparse set $T \subset O_v(u)$, any $u'$ and any $m \leq |T \cap O_w(u')|$, 
$$| \{ (u, S) : S \text{ is left $F$-sparse}, \; |S| = m, \; S \subseteq T, \; u' = R_u^{v\rightarrow w}(S) \} | \leq  {|T \cap O_w(u')| \choose m} . $$
\end{lemma}

The following lemma will be useful for computing topological pressure. 

\begin{lemma}\label{S_n pressure computation} Let $\mu_\phi$ be an ergodic equilibrium measure for any $\phi \in C(X)$. For any tempered F\o lner sequence $ \{ F_n \}$, if $S_n \subset L_{F_n}(X)$ such that $\mu_\phi(S_n) \rightarrow 1$, then 
$$P_{top}(\phi) = \lim_{n \rightarrow \infty} \frac{1}{|F_n|} \log \left(  \sum_{w \in S_n} \sup_{x \in [w]} exp(\phi_{F_n} (x) ) \right) . $$
\end{lemma}

\begin{proof} Let $\epsilon > 0$ and note by definition of topological pressure above we have 
$$\limsup_{n \rightarrow \infty} \frac{1}{|F_n|} \log \left(  \sum_{w \in S_n} \sup_{x \in [w]} exp(\phi_{F_n} (x) ) \right) \leq   P_{top}(\phi) . $$

For every $n$ define 
$$T_n = \left\{ w \in L_{F_n}(X) : \mu_\phi([w]) < exp \left(  |F_n| \left(\sup_{x \in [w]} \phi_{F_n}(x) - (P_{top}(\phi) - \epsilon) \right) \right) \right\} . $$

Note since $F_n$ is a tempered F\o lner sequence and $\mu_\phi$ is ergodic, we can apply the pointwise ergodic theorem and the Shannon-Macmillan-Breiman theorem proven in \cite{ornsteinweiss} and \cite{Gurevich} to see for $\mu_\phi$-a.e. $x \in X$, 
$$P_{top}(\phi) = h(\mu_\phi) + \int \phi d\mu_\phi = \lim_{n \rightarrow \infty} |F_n|^{-1} \left( \phi_{F_n}(x) - \log \mu_\phi([x_{F_n} ]) \right) . $$

By the definition of $T_n$ it therefore follows that $\mu_\phi( \bigcup_{N \in \N} \bigcap_{n \geq N} T_n) = 1$, and so $\mu_\phi(T_n) \rightarrow 1$. Further, $\mu( S_n \cap T_n) \rightarrow 1$, and by definition of $T_n$: 
$$\sum_{w \in S_n \cap T_n} \sup_{x \in [w]} exp \left( \phi_{F_n}(x) \right) \geq \mu_\phi(S_n \cap T_n) exp \left( |F_n|( P_{top}(\phi) -\epsilon) \right) . $$
By taking sufficiently large $n$ we therefore have 
$$\sum_{w \in S_n } \sup_{x \in [w]} exp \left( \phi_{F_n}(x) \right) \geq 0.5 exp \left( |F_n|( P_{top}(\phi) -\epsilon) \right) . $$
Since $\epsilon > 0$ was arbitrary we are finished. 
\end{proof}

We remind the reader that for $v, w \in L_F(X)$, $\xi_{v, w}: X \rightarrow X$ is the map that swaps $v$ and $w$ in the $F$ location whenever legal. 

\begin{obs}\label{infinite sum bound} If $\phi$ has summable variation according to the exhaustive sequence $\{ E_n \}$, then for any $F \Subset G$, $v, w \in L_F(X)$, and any $x \in [v]$, 
$$\sum_{g \in G} \left| \phi(\sigma_g(x)) - \phi(\sigma_g(\xi_{v, w}(x))) \right|  \leq  |F| \cdot ||\phi||_{SV_{\{ E_n\}}}. $$
\end{obs}

We note the proof of Observation \ref{infinite sum bound} is standard and the statement appears nearly identically as Proposition 3.1 of \cite{barbieri-gibbsian}. The final lemma in this section is a Stirling approximation that will be necessary to compute a lower bound on pressure in the proof Theorem \ref{sv theorem}. 

\begin{lemma}\label{stirling} For $b,  a \in \Q_+$, and any sequence $n_k \in \N$ such that for all $k$, $k!$ divides $n_k$. Let $D \in \R$ satisfying $\log \frac{a}{b} > D$. Then, for sufficiently small $c \in \Q_+$, we have
$$ \lim_{k \rightarrow \infty} n_k^{-1} \left( \log { an_k \choose c n_k} - \log {b n_k \choose c n_k} \right) > cD .$$
\end{lemma}

\begin{proof} We define for all $c \in (0, \min \{ a, b\} ) \cap \Q$, 
$$f(c) = \lim_{k \rightarrow \infty} n_k^{-1} \left( \log { an_k \choose c n_k} - \log {b n_k \choose c n_k} \right) .$$

First note we have for all $k \in \N$, 
$$\left( \log { an_k \choose c n_k} - \log {b n_k \choose c n_k} \right) = \log (a n_k)! + \log ((b-c)n_k)! - \log (b n_k)! - \log ((a-c)n_k)! . $$
If we divide by $n_k$ and take the limit as $k$ goes to infinity, Stirling's approximation implies: 
% $$f(c) \sim  a  \log a n_k - a  + (b-c) \log (b-c)n_k - (b-c) - b \log b n_k + b - (a-c) \log (a-c)n_k + (a-c)$$
$$f(c) = a \log \frac{a}{a-c}  + b \log \frac{b-c}{b} + c \log \frac{a-c}{b-c} .$$

We now examine 
$$f'(c) = \frac{a}{a-c} - \frac{b}{b-c} + \log \frac{a-c}{b-c} + c \frac{b-a}{(a-c)(c-b)} . $$
Note here $f'(0) = \log \frac{a}{b} > D$. Since $f(0) = 0$, it immediately follows that for sufficiently small $c \in \Q_+$, we know $f(c) > cD$. 
\end{proof}

\subsection{Proof of Theorem \ref{sv theorem}}

\setcounter{theorem}{0}
\begin{theorem} Let $G$ be a countable amenable group and $X$ be a $G$-subshift. Let $\phi \in SV(X)$, $\mu_\phi$ an equilibrium state for $\phi$, $F \Subset G$ and $v, w \in L_F(X)$. If $E_X(v) \subset E_X(w)$ then 
$$\mu_\phi([v]) \leq \mu_\phi([w]) \cdot  \sup_{x \in [v]} exp\left(  \sum_{g \in G} \phi(\sigma_g(x)) - \phi(\sigma_g(\xi_{v, w}(x)))  \right).$$
\end{theorem}

First we let $X$, $\phi$, $F$, $v$ and $w$ be in as the statement of the theorem. By ergodic decomposition it is sufficient to show the desired result for ergodic equilibrium states, so we let $\mu_\phi$ be an ergodic equilibrium state for $\phi$. Fix $C =  \sup_{x \in [v]}  \sum_{g \in G} \phi(\sigma_g(x)) - \phi(\sigma_g(\xi_{v, w}(x)))$ and note by Observation \ref{infinite sum bound} we know $-\infty < C < \infty$. We now suppose for a contradiction that $\mu_\phi([v]) > \mu_\phi([w]) \cdot  e^C$.

We now take $\delta \in (0, \frac{4 }{5} \mu_\phi([v]))  $ satisfying: 
$$e^{-C} \cdot  \frac{\mu_\phi([v]) - 5 \delta / 4}{\mu_\phi([w]) + 2 \delta} > 1 .$$
Note here this is attainable for all sufficiently small $\delta$ since we know 
$$f(\delta ) =  e^{-C} \cdot \frac{\mu_\phi([v]) - 5 \delta / 4}{\mu_\phi([w]) + 2 \delta} $$
is continuous and by assumption $f(0) > 1$.

Let $F_n$ be a F\o lner sequence such that for each $n$, $n!$ divides $|F_n|$ and $F_n^{-1} = F_n$. It is noted by Xu and Zheng in  \cite{xuzheng} that we can assume $F_n$ is tempered by passing to a subsequence. Re-index this sequence by $n$, and thus since $F_n$ is tempered, we know it satisfies the requirements for the Lemmas above.

We now define  $S_n \subset L_{F_n}(X)$ as follows: 
$$S_n := \{ u \in L_{F_n}(X) : |O_v(u)| \geq |F_n| (\mu_\phi([v]) - \delta) \text{ and }|O_w(u)| \leq |F_n| (\mu_\phi([w]) + \delta) \} . $$

We note here since $\mu_\phi$ is ergodic, $\mu_\phi ( S_n) \rightarrow 1$ and we can therefore apply Lemma \ref{S_n pressure computation} to see 
$$P_{top}(\phi) = \lim_{n \rightarrow \infty} |F_n|^{-1} \log \sum_{u \in S_n} \sup_{x \in [u]} exp( \phi_{F_n}(x) ) . $$

We now apply Lemma \ref{almost partition} with respect to $F$ and $\epsilon = \delta / 8$, re-index our F\o lner sequence according to the lemma, and let $\mathcal{P}  = \{ P_i : 1 \leq i \leq N \}$ be our collection of pairwise disjoint, left $F$-sparse sets. We also pass to the tail of the sequence of F\o lner sets ensuring that for each $n$ we have  
$$\frac{|\bigcup_{i \leq N} P_i \cap F_n|}{|F_n|} \geq 1 - \frac{\delta}{4} . $$

For each $1 \leq i \leq N$, we now define 
$$S_n^{i} = \{ u \in S_n : |O_v(u) \cap P_i| \geq \frac{\mu_\phi([v]) - 5 \delta / 4}{N} |F_n| \} . $$

\begin{lemma} For each $n \in \N$, $S_n = \bigcup_{1 \leq i \leq N} S_n^i$. 
\end{lemma}
 
\begin{proof} Suppose there exists some $u \in S_n \backslash \bigcup_{1 \leq i \leq N} S_n^{i}$. We therefore know for all $P \in \mathcal{P}$, $|O_v(u) \cap P| < \frac{\mu_\phi([v]) - 5 \delta / 4}{N} |F_n| $. Let $Q = F_n \backslash \bigcup_{i \leq N} P_i$ and note that 
$$|O_v(u)| = \sum_{i=1}^N | O_v(u) \cap P_i| + |O_v(u) \cap Q| < (\mu_\phi([v]) - 5 \delta / 4) |F_n| + |O_v(u) \cap Q| . $$
However, since $u \in S_n$ we know $(\mu_\phi([v]) - \delta) \cdot |F_n| < |O_v(u)|$. It follows that 
$$(\mu_\phi([v]) - \delta) \cdot |F_n| <  (\mu_\phi([v]) - 5 \delta / 4) |F_n| + |O_v(u) \cap Q| . $$
By our construction of $\mathcal{P}$ and since we took a sufficiently long tail for $F_n$, we know that $ |O_v(u) \cap Q| \leq |Q \cap F_n| \leq \frac{\delta}{4}|F_n|$. We therefore have $\mu_\phi([v]) - \delta < \mu_\phi([v]) -\delta$, arriving at a contradiction and we can conclude that $S_n = \bigcup_{1 \leq i \leq N} S_n^{i}$. 
\end{proof}

\begin{lemma}\label{S_i' pressure} There exists some fixed $1 \leq i \leq N$ and a subsequence $(n_k)$ for which we have: 
$$P_{top}(\phi) = \lim_{k \rightarrow \infty} |F_{n_k}|^{-1} \log \sum_{u \in S_{n_k}^i} \sup_{x \in [u]} exp( \phi_{F_{n_k}}(x) ) . $$
\end{lemma}

\begin{proof}  First notice since $S_n = \bigcup_{i \leq N} S_n^i$, 
$$\sum_{ u \in S_n} \sup_{x \in [u]} exp ( \phi_{F_n}(x)) \leq \sum_{i=1}^N \sum_{ u \in S_n^{i}} \sup_{x \in [u]} exp ( \phi_{F_n}(x)),$$
and we therefore have for each $n$, some $1 \leq i_n \leq N$ such that 
$$\frac{1}{N} \sum_{u \in S_n} \sup_{x \in [u]} exp ( \phi_{F_n}(x)) \leq \sum_{ u \in S_n^{i_n}} \sup_{x \in [u]} exp ( \phi_{F_n}(x)) . $$
We now use pigeonhole principle to find a subsequence $n_k$ such that $i_{n_k} = i$ is constant.  Note by our construction we have 
$$\sum_{u \in S_{n_k}^i} \sup_{x \in [u]} exp(\phi_{F_n}(x)) \geq N^{-1} \sum_{u \in S_n}\sup_{x \in [u]} exp(\phi_{F_n}(x)) . $$

We now examine 
$$P_{top}(\phi) = \lim_{n \rightarrow \infty} |F_n|^{-1} \sum_{u \in S_n} \sup_{x \in [u]} exp(\phi_{F_n}(x))  $$
$$\geq \lim_{n \rightarrow \infty} |F_n|^{-1} \log \sum_{u \in S_{n_k}^i} \sup_{x \in [u]} exp(\phi_{F_n}(x)) \geq \lim_{n \rightarrow \infty} |F_n|^{-1} \log  N^{-1} \sum_{u \in S_n}\sup_{x \in [u]}  exp(\phi_{F_n}(x))$$
$$= \lim_{n \rightarrow \infty} |F_n|^{-1} \log N^{-1} + \lim_{n \rightarrow \infty} |F_n|^{-1} \sum_{u \in S_n} \sup_{x \in [u]} exp(\phi_{F_n}(x)) = P_{top}(\phi) .$$
\end{proof}

We now apply Lemma \ref{S_i' pressure} and fix the resulting $1 \leq i \leq N$ and reindex along the resulting subsequence $(n_k)$. Denote $S_k' = S_{n_k}^i$ and fix $P = P_i$. We can therefore compute topological pressure by restricting to the sequence of collections of finite configurations $S_n'$. We will now make replacements of $v$ with $w$ in every $u \in S_n'$ at locations in $P$ with a small frequency in order to increase the pressure of the dynamical system to arrive at our contradiction.

We let $a, b \in \Q_+$ such that $a \leq (\mu_\phi([v]) - 5\delta / 4) /N$, $b \geq  (\mu_\phi([w]) + 2\delta) / N $, and $\log \frac{a}{b} > C$. Note such $a, b \in \Q_+$ must exist since $\log \frac{ \mu_\phi([v]) - 5\delta / 4}{\mu_\phi([w]) + 2\delta} > C$ by assumption. We now let $\epsilon \in \Q_+$ be sufficiently small to satisfy Lemma \ref{stirling} with respect to $a, b$ and $D = C$.

Define for each $u \in S_n'$: 
$$A_u = \{ R_u^{v \rightarrow w}(S) : S \subset O_v(u) \cap P \text{ and } |S| =  \epsilon \cdot |F_n|   \} . $$
Since $\epsilon \in \Q$ and $n!$ divides $|F_n|$,  for sufficiently large $n$,  $ \epsilon \cdot |F_n|  \in \N$ and $A_u$ is well defined. We restrict our consideration for all $n$ large enough that this definition makes sense. Additionally, since $E_X(v) \subset E_X(w)$, we have $A_u \subset L(X)$.

We can now define for each $n \in \N$, $L_n = \bigcup_{u \in S_n'} A_u$.

\begin{lemma}\label{Ln lower bound lemma} For each $n \in \N$,  $\sum_{u \in L_n} \sup_{x \in [u]} exp \left( \phi_{F_n}(x) \right)$ is bounded below by 
\begin{equation}\label{sn' lower bound}
\frac{\left( \sum_{u \in S_n'} \sup_{x \in [u]} exp \left( \phi_{F_n}(x) - \frac{ \epsilon \cdot |F_n| }{ N}  C \right) \right)  { \lceil (\mu_\phi([v]) - 5\delta / 4) |F_n| /N \rceil \choose  \epsilon \cdot |F_n|  }}{ { \lfloor (\mu_\phi([w]) + 2\delta) |F_n|/ N \rfloor \choose \epsilon \cdot |F_n|  }} .
\end{equation}
\end{lemma}

\begin{proof} Note here for each $u \in S_n'$: 
$$|A_u| = {|O_v(u) \cap P| \choose  \epsilon \cdot |F_n|   } \geq   { \lceil (\mu_\phi([v]) - 5\delta / 4) |F_n| /N \rceil \choose  \epsilon \cdot |F_n|   }  . $$

On the other hand, for every $u' \in \bigcup_{u \in S_n} A_u$, we know $u'$ came from some $u \in S_n$ by replacing $\epsilon \cdot |F_n| / N$ $v$'s with $w$. We can therefore get a bound on the following: 
$$|O_w(u') \cap P| \leq |O_w(u) \cap P| + \frac{\epsilon \cdot |F_n|}{N} < \frac{(\mu_\phi([w]) + \delta) |F_n| }{N} + \frac{\delta |F_n|}{|FF^{-1}| \cdot N} . $$
Since $|FF^{-1}| \geq 1$, we therefore have 
$$|O_w(u') \cap P| < \frac{(\mu_\phi([w]) + 2\delta) |F_n|}{N} . $$
We therefore know for any fixed $u' \in A_u$, there are at most 
$${ \lfloor (\mu_\phi([w]) + 2\delta) |F_n|  /N \rfloor \choose \epsilon \cdot |F_n| }$$
$u$s for which $u' \in A_u$.

It follows that $\sum_{u \in L_n} \sup_{x \in [u]} exp \left( \phi_{F_n}(x) \right)$ is bounded below by 
\begin{equation}\label{ln bound}
\frac{\left(  \sum_{u \in S_n'} \sum_{u' \in A_u} \sup_{x \in [u']} exp \left( \phi_{F_n}(x) \right) \right)  { \lceil (\mu_\phi([v]) - 5\delta / 4) |F_n| /N \rceil \choose  \epsilon \cdot |F_n|   }}{ { \lfloor (\mu_\phi([w]) + 2\delta) |F_n|/ N \rfloor \choose \epsilon \cdot |F_n|  }} . 
\end{equation}

$\;$

We now bound for any fixed $u \in S_n'$ and each $u' \in A_u$, $\sup_{x \in [u']} \phi_{F_n}(x)$. First note $u' = R_u^{v \rightarrow w}(S)$ for some $S \subset F_n$. Since $E_X(v) \subset E_X(w)$ and by our choice of $S$, we know $E_X(u) \subset E_X(u')$. We therefore know $\xi_{u, u'}$ (as defined above) is injective and non-identity on $[u]$. In particular, we are replacing our $v$'s located at $S$ with $w$'s just as we did in the construction of $u'$. Note here $\xi_{u, u'}([u]) \subset [u']$ and thus we have 
$$\sup_{x \in [u']} \phi_{F_n}(x) \geq \sup_{x \in \xi_{u, u'}([u])} \phi_{F_n}(x) =  \sup_{x \in [u]} \phi_{F_n}(\xi_{u, u'}(x)) . $$
We now fix any $x \in [u]$. Note that since $S$ is finite and left $F$-sparse, we may make the $v$ to $w$ replacements sequentially. Each of these replacements take the form $\sigma_{g^{-1}} ( \xi_{v, w}( \sigma_g (x)) $ for some $g \in G$. Note here for each replacement of this kind, by definition of $C$ we have 
$$\phi_{F_n}( \sigma_{g^{-1}} ( \xi_{v, w}( \sigma_g (x)) ) \geq \phi_{F_n}( x) - C . $$
Since we will make $|S|$ of these replacements, it follows that 
$$\phi_{F_n}( \xi_{u, u'}(x)) \geq \phi_{F_n}(x) - |S| C . $$
We therefore know that 
$$\sum_{u \in S_n'} \sum_{u' \in A_u} \sup_{x \in [u']} exp \left( \phi_{F_n}(x) \right) \geq \sum_{u \in S_n'} |A_u| \sup_{x \in [u]} exp \left( \phi_{F_n}(x)  - |S| C\right) . $$
Since $|A_u| \geq 1$, we therefore have 
\begin{equation}\label{sum bound}
\sum_{u \in S_n'} \sum_{u' \in A_u} \sup_{x \in [u']} exp \left( \phi_{F_n}(x) \right) \geq \sum_{u \in S_n'} \sup_{x \in [u]} exp \left( \phi_{F_n}(x)  - |S| C\right). 
\end{equation} 
Combining (\ref{ln bound}) and (\ref{sum bound}), we arrive at our desired result that  $\sum_{u \in L_n} \sup_{x \in [u]} exp \left( \phi_{F_n}(x) \right)$ is bounded below by (\ref{sn' lower bound}). 
\end{proof}

We will conclude the proof of Theorem \ref{sv theorem} by computing a lower bound on topological pressure and arrive at a contradiction. Since $L_n \subset L_{F_n}(X)$, it must be the case that 
$$P_{top}(\phi) \geq \limsup_{n \rightarrow \infty} |F_n|^{-1} \log \left( \sum_{u \in L_n} \sup_{x \in [u]} exp \left( \phi_{F_n}(x) \right) \right) . $$
By application of Lemma \ref{Ln lower bound lemma}, we can see that $P_{top}(\phi)$ is bounded below by:  
$$\liminf_{n \rightarrow \infty} |F_n|^{-1} \log \frac{\left( \sum_{u \in S_n'} \sup_{x \in [u]} exp \left( \phi_{F_n}(x) - \frac{ \epsilon \cdot |F_n| }{ N}  C \right) \right)  { \lceil (\mu_\phi([v]) - 5\delta / 4) |F_n| /N \rceil \choose  \epsilon \cdot |F_n|  }}{ { \lfloor (\mu_\phi([w]) + 2\delta) |F_n|/ N \rfloor \choose \epsilon \cdot |F_n|  }} . $$
Which is exactly 
\begin{equation}\label{bound1} = \liminf_{n \rightarrow \infty} |F_n|^{-1} \log \left( \sum_{u \in S_n'} \sup_{x \in [u]} exp \left( \phi_{F_n}(x)    \right) \right)   - \epsilon \cdot C\end{equation} 
$$ + |F_n|^{-1} \log  { \lceil (\mu_\phi([v]) - 5\delta / 4) |F_n| /N  \rceil \choose  \epsilon \cdot |F_n| }  - |F_n|^{-1} \log   {\lfloor(\mu_\phi([w]) + 2\delta) |F_n|/ N \rfloor \choose\epsilon \cdot |F_n| } . $$
We now note by our choice of $a$, $b$, and $\epsilon$, 
$$\lim_{n \rightarrow \infty} |F_n|^{-1}   \log  { \lceil (\mu_\phi([v]) - 5\delta / 4) |F_n| /N \rceil \choose  \epsilon \cdot |F_n|   }  -|F_n|^{-1}  \log  { \lfloor (\mu_\phi([w]) + 2\delta) |F_n|/ N \rfloor \choose \epsilon \cdot |F_n|  }$$
$$\geq  \lim_{n \rightarrow \infty} |F_n|^{-1}   \log  { a |F_n| \choose  \epsilon \cdot |F_n|   }  -|F_n|^{-1}  \log  { b |F_n| \choose \epsilon \cdot |F_n|  } > \epsilon C .$$
It therefore follows that Equation \ref{bound1} is strictly greater than 
$$\liminf_{n \rightarrow \infty} |F_n|^{-1} \log \left( \sum_{u \in S_n'} \sup_{x \in [u]} exp \left( \phi_{F_n}(x)    \right) \right) , $$ 
and since Equation \ref{bound1} is a lower bound for pressure, we therefore know 
$$P_{top}(\phi) > \liminf_{n \rightarrow \infty} |F_n|^{-1} \log \left( \sum_{u \in S_n'} \sup_{x \in [u]} exp \left( \phi_{F_n}(x)    \right) \right) . $$
This contradicts Lemma \ref{S_i' pressure}. In particular, this means our assumption that $\mu_\phi([v]) > \mu_\phi([w]) \cdot e^C$ was incorrect, arriving at our desired result that if $E_X(v) \subset E_X(w)$, then 
$$\mu_\phi([v]) \leq \mu_\phi([w]) \cdot  \sup_{x \in [v]} exp\left(  \sum_{g \in G} \phi(\sigma_g(x)) - \phi(\sigma_g(\xi_{v, w}(x)))  \right). $$

\subsection{Locally constant corollaries}

We will now explore the case where $\phi$ is locally constant. 

The following result shows that when the finite configurations $v$ and $w$ agree on a sufficiently chosen boundary, we may compute bounds for their relative measure using only information from the configurations themselves. Further, when they have equal extender sets, we have a closed form formula depending only on $\phi$, $v$, and $w$ to compute the ratio of their measures.

We will say that $\phi$ is an {\bf $H$-potential} for some $H \Subset G$ if for all $x, y \in X$ such that $x_H = y_H$, then $\phi(x) = \phi(y)$, i.e. $Var_H(\phi) = 0$. Additionally, for any $H \subset G$, we will denote $H^\pm = H \cup H^{-1}$.

\begin{corollary} Let $H, F \Subset G$, $v, w \in L_F(X)$, and $\phi$ be an $H$-potential. Suppose that $E_X(v) \subset E_X(w)$ and for all $g \in F^c H^\pm \cap F$, $v_g = w_g$ (i.e. $v$ and $w$ agree on their $H$-boundary). Then for any equilibrium state $\mu_\phi$ for $\phi$, 
$$\mu_\phi([v]) \leq \mu_\phi([w])  \cdot exp \left( \sum_{g \in F \backslash F^c H^{-1}} \phi ( \sigma_g (v)) - \phi ( \sigma_g(\xi_{v, w}(w))) \right).$$
\end{corollary}

\begin{proof} We first note, for all $x \in [v]$ and for any $g \in F \cap F^c H^{-1}$, by assumption we have $x_{gH} = \xi_{v, w}(x)_{gH}$. Since $\phi$ is an $H$-potential, it immediately follows that for all $g \in F \cap F^c H^{-1}$, 
$$ \phi ( \sigma_g (x)) - \phi ( \sigma_g(\xi_{v, w}(x))) = 0 . $$
We now consider $g \in F^c$. If $gH \cap F = \emptyset$, then $x_{gH} = \xi_{v, w}(x)_{gH}$ by definition of $\xi_{v, w}$. We now consider $gH \cap F \neq \emptyset$. However, by assumption since $g \in F^c$, we know for all $f \in gH \cap F$, $v_f = w_f$ and again we can conclude $x_{gH} = \xi_{v, w}(x)_{gH}$. Thus we know for all $g \in F^c$, 
$$ \phi \circ \sigma_g (x) - \phi \circ \sigma_g(\xi_{v, w}(x)) = 0 . $$
Since $\phi$ is locally constant, it has summable variation, and we can apply Theorem \ref{sv theorem} to see 
$$\mu_\phi([v]) \leq \mu_\phi([w]) \sup_{x \in [v]} exp \left( \sum_{g \in F \backslash F^c H^{-1}} \phi ( \sigma_g (x)) - \phi ( \sigma_g(\xi_{v, w}(x))) \right) . $$
We let $x, y \in [v]$, $g\in  F \backslash F^c H^\pm$, and $h \in H$. By construction we know $g \notin F^c H^{-1}$, and thus $gH \cap F^c = \emptyset$. Since $\phi$ is an $H$-potential, it follows that $\phi \circ \sigma_g (x) = \phi \circ \sigma_g (y)$, and similarly for $\xi_{v, w}(x)$ and $\xi_{v, w}(y)$. We can therefore conclude that $\sum_{g \in F \backslash F^c H^\pm} \phi ( \sigma_g (x)) - \phi ( \sigma_g(\xi_{v, w}(x)))$ does not depend on choice of $x \in [v]$ and can be computed by only looking at the cylinder set, and we have arrived at our desired result. 
\end{proof}

As an immediate corollary we have: 

\begin{cor} Let $H, F \Subset G$, $v, w \in L_F(X)$, and $\phi$ be an $H$ potential. Suppose that $E_X(v) = E_X(w)$ and for all $g \in F^c H^\pm \cap F$, $v_g = w_g$ (i.e. $v$ and $w$ agree on their $H$-boundary). Then for any equilibrium state $\mu_\phi$ for $\phi$, 
$$\frac{\mu_\phi([v])}{exp \left( \sum_{g \in F \backslash F^c H^{-1}} \phi \circ \sigma_g (v) \right)} = \frac{\mu_\phi([w])}{exp \left( \sum_{g \in F \backslash F^c H^{-1}} \phi \circ \sigma_g (w) \right)} . $$
\end{cor}

We remind the reader that in the case of $\phi = 0$, equilibrium states correspond with measures of maximal entropy. In this case, as a corollary of Theorem \ref{sv theorem} we extend the results of Garcia-Ramos and Pavlov in \cite{GRP} to the countable amenable subshift setting. 

\begin{cor} Let $X$ be a $G$-subshift, $F \Subset G$, $v, w \in L_F(X)$, and $\mu$ a measure of maximal entropy. If $E_X(v) \subset E_X(w)$, then $\mu([v]) \leq \mu([w])$.  
\end{cor}

%
%
%{\bf  REMOVE THIS OR JUST CLEAN IT UP SO YOU DON'T REPEAT IT IN THE INTRO???}
%
%In the case where $\phi = 0$, equilibrium states correspond exactly to measures of maximal entropy. In particular, we can extend the following theorem by Garcia-Ramos and Pavlov to arbitrary countable amenable groups. 
%\begin{fact}[Theorem 4.4 \cite{GRP}] Let $X \subset \mathcal{A}^{\Z^d}$ be a subshift, and let  $\mu \in \mathcal{M}_\sigma(X)$ be a measure of maximal entropy. Then for any $F \Subset \Z^d$, $v, w \in L_F(X)$, if $E_X(v) \subset E_X(w)$ then, 
%$$\mu([v]) \leq \mu([w]) . $$
%\end{fact}
%
%
%

%In the case where $\phi = 0$, equilibrium states correspond exactly to measures of maximal entropy. In particular, we can extend the following theorem by Garcia-Ramos and Pavlov to arbitrary countable amenable groups. 
%
%\begin{fact}[Theorem 4.4 \cite{GRP}] Let $X \subset \mathcal{A}^{\Z^d}$ be a subshift, and let  $\mu \in \mathcal{M}_\sigma(X)$ be a measure of maximal entropy. Then for any $F \Subset \Z^d$, $v, w \in L_F(X)$, if $E_X(v) \subset E_X(w)$ then, 
%$$\mu([v]) \leq \mu([w]) . $$
%\end{fact}

\section{Conformal Gibbs Results}

\subsection{Homoclinic relation and conformal Gibbs} As previously mentioned, a measure $\mu$ is Gibbs for a potential $\phi$ (in the sense of the theorem of Dobru\v sin and that of Lanford-Ruelle) if and only if it is conformal Gibbs for $\phi$. This section will explore our conformal Gibbs-like result, for which we will follow the definition found in \cite{conformalisgibbs}. 

First we define the {\bf homoclinic relation}, also known as the {\bf Gibbs relation}, as $\mathfrak{T}_X \subset X \times X$ such that $(x, y) \in \mathfrak{T}_X$ if and only if $x_{F^c} = y_{F^c}$ for some $F \Subset G$. We now say a Borel isomorphism $f: X \rightarrow X$ is a {\bf holonomy of $\mathfrak{T}_X$} if for all $x \in X$, $(x, f(x)) \in \mathfrak{T}_X$. It is known (see \cite{conformalisgibbs}) that there exists a countable group $\Gamma$ of holonomies of $X$ such that $\mathfrak{T}_X = \{ (x, \gamma(x)) : x \in X, \gamma \in \Gamma \}$. In fact, by a trivial extension of Lemma 1 of \cite{conformalisgibbs} we can see immediately that $\Gamma = \braket{ \xi_{v, w} : v, w \in L_F(X), F \Subset G }$ is a countable group of holonomies such that $\mathfrak{T}_X = \{ (x, \gamma(x)) : x \in X, \gamma \in \Gamma \}$. 

For any measurable $A \subset X$, we define $\mathfrak{T}_X(A) = \bigcup_{x \in A} \{ y \in X : (x, y) \in \mathfrak{T}_X \}$. For a Borel measure $\mu$ on $X$, we say {\bf $\mu$ is $\mathfrak{T}_X$-nonsingular} if for all null sets $A \subset X$, $\mu(\mathfrak{T}_X(A)) = 0$. Note for any holonomy $f:X \rightarrow X$ if $A \subset X$ is a null set and $\mu$ is  $\mathfrak{T}_X$-nonsingular, then $\mu \circ f(A) \leq \mu(\mathfrak{T}_X(A)) = 0$ and therefore $\mu \circ f << \mu$. 

We define a {\bf cocycle on $\mathfrak{T}_X$} to be any measurable $\psi : \mathfrak{T}_X \rightarrow \R$ such that for all $x, y, z \in X$ with $(x,y), (y, z) \in \mathfrak{T}_X$, we have $\psi(x, y) + \psi(y, z) = \psi(x, z)$. We can now define conformality following Definition 1 of \cite{conformalisgibbs}: 
\begin{defn}  Let $\mu$ be a $\mathfrak{T}_X$-nonsingular Borel probability measure on $X$ and let $\psi : \mathfrak{T}_X \rightarrow \R$ be a cocycle. We say {\bf$\mu$ is $(\psi, \mathfrak{T}_X)$-conformal} if for any holonomy $f: X \rightarrow X$, we have
$$\frac{ d\mu \circ f}{d\mu}(x) = exp\left( \psi(x, f(x) \right) . $$
\end{defn}

Since we are concerned with Gibbs measures, in general for a potential $\phi \in SV(X)$ we define the following cocycle: 
$$\psi_\phi (x, y) = \sum_{g \in G} \phi \circ \sigma_g (y) - \phi \circ \sigma_g(x)$$
and we will say that a measure $\mu$ is {\bf conformal Gibbs} for $\phi$ if it is $(\psi_\phi, \mathfrak{T}_X)$-conformal. 

Connecting this concept further to the classical notion of Gibbs measures, it was shown in  \cite{conformalisgibbs} that a measure is conformal Gibbs for $\phi$ if and only if it is Gibbs for $\phi$ in the sense of Definition \ref{dlr gibbs}.

In addition to the homoclinic relation, in \cite{meyerovitch} Meyerovitch defined a subrelation as follows. First let 
$$\mathcal{F}(X)= \{ f \in Homeo(X) : \exists F \Subset G \; s.t. \; \forall x \in X,  \; f(x)_{F^c} = x_{F^c} \}.$$ 
In particular, $\mathcal{F}(X)$ is the group of all background preserving homeomorphisms. We then define $\mathfrak{T}_X^0 = \{ (x, f(x)) : x \in X, f \in \mathcal{F}(X) \}$. Clearly $\mathfrak{T}_X^0 \subset \mathfrak{T}_X$, however in general these are not necessarily equal. 

Notice that when $X$ is an SFT, then each $\xi_{ v, w}$ is in fact a homeomorphism and thus $\xi_{v, w} \in \mathcal{F}(X)$. In this case $\mathfrak{T}_X^0 = \mathfrak{T}_X$. Additionally, even if $X$ is not an SFT, if $E_X(v) = E_X(w)$, $\xi_{v, w}$ is again a homeomorphism. However, it need not be the case that $\xi_{ v, w}$ is continuous in general. 

\begin{ex}\label{sunnysideexample} Take $X \subset \{ 0, 1\}^\Z$ to the orbit closure of $x = 0^\infty 1 0^\infty$ (which is sometimes referred to as the sunny side up subshift). We define $\xi_{1, 0}$ to swap a $1$ and $0$ in the identity location and note here that $E_X(1) \subsetneq E_X(0)$. For each $n \geq 0$, define $x_n \in X$ such that $(x_n)_n = 1$. Then for all $n \geq 1$, $\xi_{1, 0}(x_n) = x_n$. However, we have $x_n \rightarrow 0^\infty$ and $\xi_{1, 0}(0^\infty) = x_0$. It immediately follows that 
$$\lim_{n \rightarrow \infty} \xi_{1, 0}(x_n) = 0^\infty \neq x_0 = \xi_{1, 0}( \lim_{n \rightarrow \infty} x_n) $$
and we can conclude $\xi_{1, 0}$ is not continuous. 
\end{ex}

In the case of $G = \Z^d$ and $\phi \in C(X)$ has $d$-summable variation, Meyerovitch proved the following result: 

\begin{fact}[Meyerovitch \cite{meyerovitch}] Let $X \subset \mathcal{A}^{\Z^d}$ be a subshift, $\phi \in C(X)$ a potential with $d$-summable variation, and $\mu_\phi$ be an equilibrium state for $\phi$ that  $\mathfrak{T}_X$-nonsingular. Then $\mu_\phi$ is $(\psi_\phi, \mathfrak{T}_X^0)$-conformal.
\end{fact}

In particular, $\mu_\phi$ is conformal Gibbs with respect to the sub-relation $\mathfrak{T}_X^0$. What this means in the terminology of our paper is for $G = \Z^d$ and $\phi$ with $d$-summable variation, if $E_X(v) = E_X(w)$, then for $\mu_\phi$ almost every $x \in X$, $\frac{d (\mu \circ \xi_{v, w} ) }{d\mu}(x) = exp \left( \psi_\phi(x, \xi_{v, w}(x)) \right)$. We will use Theorem \ref{sv theorem} to extend this result both to the general amenable group setting as well as providing an inequality on the derivative when $E_X(v) \subset E_X(w)$.

\subsection{Proof of conformal Gibbs result}

In this section, we will prove the following theorem: 

\begin{theorem} Let $F \Subset G$, $v, w \in L_F(X)$, $\phi \in SV(X)$, and $\mu_\phi$ an equilibrium state for $\phi$. If $E_X(v) \subset E_X(w)$, then $\mu_\phi \circ \xi_{v, w}$ is absolutely continuous with respect to $\mu_\phi$ when restricted to $[w]$ and for $\mu_\phi$-almost every $x \in [w]$, 
$$\frac{ d (\mu \circ \xi_{v, w} ) }{d \mu}(x) \leq exp \left( \psi_\phi( x, \xi_{v, w}(x)) \right) . $$
\end{theorem}

For the remainder of this section we fix some $F \Subset G$, $v, w \in L_F(X)$, $\phi \in SV(X)$, and $\mu_\phi$ an equilibrium state for $\phi$. 

\begin{obs} If $E_X(v) \subset E_X(w)$, then $\mu_\phi \circ \xi_{v, w} << \mu_\phi$ when restricted to $[w]$. 
\end{obs}

\begin{proof} Let $A \subset [w]$ be a null set. Note here 
$$\xi_{v, w}(A) = \{ x \in A : \xi_{v, w}(x) = x\} \sqcup \{ \xi_{v, w}(x) : x \in A \text{ and }  \xi_{v, w}(x) \in [v]  \} . $$
Since $ \{ x \in A : \xi_{v, w}(x) = x\}  \subset A$, it must be a null set, and thus it is sufficient to assume without loss of generality that $\xi_{v, w}(A) \cap A = \emptyset$. 

Since we know $\mu_\phi$ is outer regular, we can let $U_n =\bigcup_{ i \leq N_n}  [w_{n, i}]$ be an outer approximation by cylinder sets such that $\mu_\phi(A) = \lim_{n \rightarrow \infty} \mu_\phi(U_n)$. By taking sufficiently large $n$, we can assume without loss of generality that $[w_{n, i}] \subset [w]$ for all $n, i$. Note for each $n, i$ we have $E_X(\xi_{v, w}(w_{n, i})) \subset E_X(w_{n, i})$. We now apply Theorem \ref{sv theorem} with respect to $\phi$, $\mu_\phi$, $v_{n, i}$ and $w_{n, i}$ to see:  
$$\mu_\phi(\xi_{v, w}(A)) = \lim_{n \rightarrow \infty} \sum_{i \leq N_n} \mu_\phi  \left( \xi_{v, w}([w_{n, i}]) \right) \leq \lim_{n \rightarrow \infty} \sum_{i \leq N_n} \mu_\phi ([w_{n, i}]) \sup_{x \in [v]} exp (f(x)) = 0 . $$
\end{proof}

In particular, this means we can discuss the LRN derivative when restricted to $[w]$. We now define for each $n \in \N$, $\phi_n : X \rightarrow \R$ to be an $E_n$ potential approximating $\phi$ from above. In particular, for all $u \in L_{E_n}(X)$, for any $x \in [u]$, let $\phi_n(x) = \sup_{y \in [u]} \phi(y)$.

\begin{obs}\label{summable var norm convergence} $\phi_n$ converges to $\phi$ with respect to the summable variation norm induced by $\{ E_n \}$. 
\end{obs}

\begin{proof} First note, for any $k, n \in \N$, we have 
$$Var_{E_k}(\phi_n - \phi) \leq 2 || \phi||_\infty \text{ and } Var_{E_k}(\phi_n - \phi) \leq 2 Var_{E_k}(\phi) .$$

Let $\epsilon > 0$ and note since $\phi$ has summable variation according to $\{ E_n \}$,  by definition there exists some $M \in \N$ such that $\sum_{k=M+1}^\infty |E_{k+1} \backslash E_k| Var_{E_k}( \phi) < \epsilon$. We fix this $M$. We now note by continuity of $\phi$ and since $\{ E_n\}$ is an exhaustive sequence, we know there exists some $N \in \N$ such that for all $n \geq N$, 
$$ || \phi_N - \phi||_\infty < \epsilon \left( 2 |E_1| + \sum_{k=1}^M |E_{k+1} \backslash E_k| \right)^{-1}. $$
We now let $n \geq N$ and we examine: 
$$|| \phi_N - \phi||_{SV_{\{ E_n \}}} = 2 |E_1| \cdot || \phi_n - \phi||_\infty + \sum_{k=1}^M |E_{k+1} \backslash E_k| Var_{E_k}(\phi_n - \phi) + \sum_{k=M+1}^\infty |E_{k+1} \backslash E_k| Var_{E_k}(\phi_n - \phi)$$
$$\leq || \phi_n - \phi||_\infty \left( 2 |E_1| + \sum_{k=1}^M |E_{k+1} \backslash E_k| \right) + 2 \sum_{k=M+1}^\infty |E_{k+1} \backslash E_k| Var_{E_k}( \phi).$$
By our choice of $N$ and $M$, we know for all $n \geq N$, $||\phi_N - \phi||_{SV_{\{ E_n \}}} <  3 \epsilon$. Since $\epsilon$ was arbitrary, we can conclude the proof of the observation. 
\end{proof}

\begin{obs}\label{f cont} If $E_X(v) \subset E_X(w)$, then the function $f(x) = \psi_\phi(x, \xi_{v, w}(x))$ is continuous on $[v]$. 
\end{obs}

\begin{proof} Using our locally constant approximations $\phi_n$, we define for each $n \in \N$:  
$$f_n(x) = \psi_{\phi_n}(x, \xi_{v, w}(x) ) = \sum_{g \in G} \phi_n(\sigma_g(\xi_{v, w}(x))) - \phi_n(\sigma_g(x)) . $$
Since $\phi_n$ is locally constant, we know the infinite sum is in fact a finite sum and so $f_n$ is continuous on $[v]$. We now claim that $f_n \rightarrow f$ uniformly on $[v]$. By Observation \ref{infinite sum bound}, we know for all $x \in [v]$, 
$$|f_n(x) - f(x)| \leq |F| \cdot ||\phi_n - \phi||_{SV_{\{ E_n \}}} . $$
Since $\phi_n \rightarrow \phi$ with respect to the summable variation norm by Observation \ref{summable var norm convergence}, we can conclude that $f_n \rightarrow f$ uniformly on $[v]$. It immediately follows that $f$ is continuous on $[v]$.  
\end{proof}

We can combine the above observations with Theorem \ref{sv theorem} to prove the desired result.  
\setcounter{theorem}{2}
\begin{theorem} Let $F \Subset G$, $v, w \in L_F(X)$, $\phi \in SV(X)$, and $\mu_\phi$ an equilibrium state for $\phi$. If $E_X(v) \subset E_X(w)$, then for $\mu_\phi$-almost every $x \in [w]$, 
$$\frac{ d\mu \circ \xi_{v, w}}{d \mu}(x) \leq exp \left( \psi_\phi( x, \xi_{v, w}(x)) \right) . $$
\end{theorem}

\begin{proof} First we fix $F \Subset G$, $v, w \in L_F(X)$, $\phi \in C(X)$, and $\mu_\phi$ as in the theorem. We again define for all $x \in X$, 
$$f(x) = \psi_\phi(x, \xi_{v, w}(x)) = \sum_{g \in G} \phi(\sigma_g(\xi_{v,w}(x))) - \phi(\sigma_g(x))$$
which we know to be continuous on $[v]$ by Observation \ref{f cont}. 

We fix any F\o lner sequence $\{ F_n \}$ and assume without loss of generality that $F \subset F_n$ for all $n$. Fix any $x \in [w]$. We define $w_n = x_{F_n}$ and $v_n = \xi_{v, w}(x)_{F_n}$. Note here $\bigcap_{n\in \N} [w_n] = x$, $\bigcap_{n \in \N} [v_n] = \xi_{v, w}(x)$. Since $F \subset F_n$, we know, $E_X(v_n) \subset E_X(w_n)$. We can directly apply Theorem \ref{sv theorem} to see for each $n \in \N$, 
$$\mu_\phi(v_n) \leq \mu_\phi(w_n) \sup_{y \in [v_n]} exp \left( - f(y) \right) . $$

Since by Observation \ref{f cont} we know $f$ is continuous on $[v]$, and thus there exists some $y_n \in [v_n]$ such that $\sup_{y \in [v_n]} exp \left( - f(y) \right) = exp \left( - f(y_n) \right)$. We therefore have 
$$\frac{\mu_\phi(v_n)}{\mu_\phi(w_n)}  \leq exp \left(  - f(y_n) \right) . $$

Since $y_n \in [v_n]$, and $\bigcap_{n \in \N} [v_n] = \{ \xi_{v, w}(x) \}$, we know $y_n \rightarrow \xi_{v, w}(x)$. We can now take limits to see
$$\frac{d \mu_\phi \circ \xi_{v, w}}{d \mu_\phi} (x) = \lim_{n \rightarrow \infty} \frac{\mu_\phi(v_n)}{\mu_\phi(w_n)} \leq \lim_{n \rightarrow \infty} exp \left( - f(y_n) \right) = exp( f(x)) . $$
\end{proof}

Finally, we note that whenever $\mu_\phi$ satisfies the equations in Theorem \ref{homoclinic theorem}, it is trivial to show $\mu_\phi$ must also satisfy the equations in Theorem \ref{sv theorem}. 

\begin{obs}\label{thm3 obs}
Let $X$ be a subshift and $u, v \in L_F(X)$ such that $E_X(v) \subset E_X(w)$. Let $\phi \in C(X)$ and $\mu_\phi$ satisfy for $\mu_\phi$-almost every $x \in [w]$, 
$$\frac{ d (\mu \circ \xi_{v, w} ) }{d \mu}(x) \leq exp \left( \psi_\phi( x, \xi_{v, w}(x)) \right). $$
Then, 
$$\mu_\phi([v]) \leq \mu_\phi([w]) \cdot \sup_{x \in [v]} exp \left( \sum_{g \in G} \phi ( \sigma_g(x)) - \phi(\sigma_g(\xi_{v, w}(x))) \right). $$ 
\end{obs}

\begin{proof} First we notice that since $E_X(v) \subset E_X(w)$, then $[v] \subseteq \xi_{v,w}([w])$. Further, since $\xi_{v, w}$ is an involution we know 
$$\mu_\phi([v]) = \mu_\phi( \xi_{v, w}( \xi_{v, w}([v]))) =  \int_{x \in \xi_{v, w}([v])} \frac{d (\mu_\phi \circ \xi_{v, w} ) }{d \mu_\phi }(x) d\mu_\phi .$$
Since by assumption, for $\mu_\phi$-a.e. $x \in [w]$, 
$$\frac{ d (\mu \circ \xi_{v, w} ) }{d \mu}(x) \leq exp \left( \psi_\phi( x, \xi_{v, w}(x)) \right), $$
we know that 
$$\int_{x \in \xi_{v, w}([v])} \frac{d (\mu_\phi \circ \xi_{v, w} ) }{d \mu_\phi }(x) d\mu_\phi \leq \int_{x \in \xi_{v, w}([v])} exp \left( \psi_\phi( \xi_{v,w}(x), \xi_{v,w}(\xi_{v, w}(x))) \right) d\mu_\phi. $$
It follows immediately that 
$$\mu_\phi([v]) \leq \mu_\phi( \xi_{v, w}([v])) \cdot \sup_{y \in [v]} exp \left( \psi_\phi( \xi_{v,w}(x), x) \right). $$
Since $\xi_{v, w}([v]) \subset [w]$, we arrive at our desired conclusion.
\end{proof}

We will now note that as a corollary of Theorem \ref{homoclinic theorem} we can extend Theorem 3.1 and Corollary 3.2 of \cite{meyerovitch} to the countable amenable group subshift setting. We have become aware that Theorem B from \cite{barbieri-meyerovitch} extends this further to the countable sofic group subshift setting.

\begin{corollary} Let $X \subset \mathcal{A}^G$ be a subshift over a countable amenable group $G$, let $\phi \in SV(X)$ be a potential with summable variation, and let $\mu_\phi$ be an equilibrium state for $\phi$. Then $\mu_\phi$ is $(\mathfrak{T}_X^0, \psi_\phi)$-conformal. 
\end{corollary}

\begin{proof} First note, as observed in \cite{meyerovitch}, the collection $\{ \xi_{v, w} :  E_X(v) = E_X(w) \}$ generates $\mathcal{F}(X)$. By application of the results in \cite{meyerovitch}, it is therefore sufficient to show that for any $v, w \in L_F(X)$ with $E_X(v) = E_X(w)$, we have for all $x \in X$ 
$$\frac{ d\mu \circ \xi_{v, w}}{d \mu}(x) = exp \left( \psi_\phi( x, \xi_{v, w}(x)) \right) . $$

Let $v, w \in L_F(X)$ such that $E_X(v) = E_X(w)$. Let $x \in [w]$ and for all $n \in \N$, define $v_n, w_n \in L_{F_n}(X)$ as in the proof of Theorem \ref{homoclinic theorem}. By application of Theorem \ref{sv theorem} we can see that 
$$\inf_{y \in [v_n]} exp  \left( - f(y) \right)  \leq \frac{\mu_\phi([v_n])}{\mu_\phi([w_n])} \leq \sup_{y \in [v_n]} exp  \left( - f(y) \right). $$

By taking limits we can conclude that for all $x \in [w]$, 
$$\frac{ d\mu \circ \xi_{v, w}}{d \mu}(x) = exp \left( \psi_\phi( x, \xi_{v, w}(x)) \right) . $$
The proof is similar for $x \in [v]$, and for all $x \in X \backslash ([v] \cup [w])$ we know 
$$\frac{ d\mu \circ \xi_{v, w}}{d \mu}(x) = exp \left( \psi_\phi( x, \xi_{v, w}(x)) \right) = 1.$$
Again, by following the techniques in \cite{meyerovitch} it follows that $\mu_\phi$ is $(\mathfrak{T}_X^0, \psi_\phi)$-conformal. 
\end{proof}

%And as an immediate corollary by combining the results in \cite{conformalisgibbs} with Corollary \ref{conformal cor}, we arrive at the following: 
%
%\begin{cor}Let $X \subset \mathcal{A}^G$ be a subshift over a countable amenable group $G$, let $\phi \in SV(X)$, and let $\mu_\phi$ be an equilibrium state for $\phi$. If $\mu_\phi$ is $\mathfrak{T}_X$-nonsingular, then $\mu_\phi$ is Gibbs for $\phi$.
%\end{cor}

\subsection*{Disclosures} The author has no conflicts of interest to declare that are relevant to the content of this article. No funding was received to assist with the preparation of this manuscript.

\bibliographystyle{plain}
\bibliography{mybib}

\end{document}